\documentclass[11pt]{article}
\usepackage{amsthm,amstext}
\usepackage{amsmath}
\usepackage{graphicx}
\usepackage{amsfonts}
\usepackage{amssymb}
\theoremstyle{plain}
\newtheorem{theorem}{Theorem}[section]

\theoremstyle{definition}
\newtheorem{definition}[theorem]{Definition}

\theoremstyle{remark}

\date{}
\title{\bf A Short Note on Fuzzy Characteristic \\Interior Ideals of Po-$\Gamma$-Semigroups}\vspace{.25 in}
\author{ {\bf Samit Kumar Majumder}\\
Tarangapur N.K High School, Tarangapur,\\
Uttar Dinajpur, West Bengal-733 129, INDIA\\
{\tt samitfuzzy@gmail.com}
 }

\begin{document}
\maketitle

\begin{abstract}

In this paper the concept of fuzzy characteristic interior ideals in partially ordered $\Gamma$-semigroups$($Po-$\Gamma$-semigroups$)$ has been introduced. It is observed that it satisfies level subset criterion and characteristic function criterion.\\

\textbf{AMS Mathematics Subject Classification(2000):}\textit{\ }08A72,
20M12, 3F55\\

\textbf{Key Words and Phrases:}\textit{\ } Po-$\Gamma$-Semigroup, Fuzzy interior ideal, Fuzzy characteristic interior ideal.
\end{abstract}

\section{Introduction}
A  semigroup is an algebraic structure consisting of a non-empty set $S$ together with an associative binary operation\cite{H}. The formal study of semigroups began in the early 20th century. Semigroups are important in many areas of mathematics, for example, coding and language theory, automata theory, combinatorics and mathematical analysis. The concept of fuzzy sets was introduced by {\it Lofti Zadeh}\cite{Z} in his classic paper in 1965. {\it Azirel Rosenfeld}\cite{R} used the idea of fuzzy set to introduce the notions of fuzzy subgroups. {\it Nobuaki Kuroki}\cite{K1,K2,K3,Mo} is the pioneer of fuzzy ideal theory of semigroups. The idea of fuzzy subsemigroup was also introduced by {\it Kuroki}\cite{K1,K3}. In \cite{K2}, {\it Kuroki} characterized several classes of semigroups in
terms of fuzzy left, fuzzy right and fuzzy bi-ideals. Others who worked on fuzzy semigroup theory, such as {\it X.Y.
Xie}\cite{X1,X2}, {\it Y.B. Jun}\cite{J}, are mentioned in the bibliography. {\it X.Y. Xie}\cite{X1} introduced the idea of extensions of fuzzy ideals in semigroups.

The notion of a $\Gamma$-semigroup was introduced by {\it Sen} and {\it Saha}\cite{Sen2} as a generalization of semigroups and ternary semigroup. $\Gamma$-semigroup have been analyzed by lot of mathematicians, for instance by {\it Chattopadhyay}\cite{C1,C2}, {\it Dutta} and {\it Adhikari}\cite{D1,D2}, {\it Hila}\cite{H1,H2}, {\it Chinram}\cite{Ch1}, {\it Saha}\cite{Sa}, {\it Sen} et al.\cite{Sen1,Sen2,Sa}, {\it Seth}\cite{Se}. {\it S.K. Sardar} and {\it S.K. Majumder}\cite{D3,D4,S1,S2} have introduced the notion of fuzzification of ideals, prime ideals, semiprime ideals and ideal extensions of $\Gamma$-semigroups and studied them via its operator semigroups. Authors who worked on fuzzy partially ordered $\Gamma$-semigroup theory are {\it Y.I. Kwon} and {\it S.K. Lee}\cite{Kw}, {\it Chinram}\cite{Ch2}, {\it P. Dheena} and {\it B. Elavarasan}\cite{Dh}, {\it M. Siripitukdet} and {\it A. Iampan}\cite{Si1,Si2,Si3}. In \cite{Kh}, {\it A.Khan, T. Mahmmod} and {\it M.I. Ali} introduced the concept of fuzzy interior ideals in Po-$\Gamma$-semigroups. In this paper the concept of fuzzy characteristic ideals of a Po-$\Gamma$-semigroup has been introduced and it is observed that it satisfies level subset criterion as well as characteristic function criterion.
\section{Preliminaries}

In this section we discuss some elementary definitions that we use
in the sequel.\\

\begin{definition}
\cite{I} Let $S$ and $\Gamma$ be two non-empty sets. $S$ is called a $\Gamma$-semigroup if there exist mapping from
$S\times\Gamma\times S$ to $S,$ written as $(a,\alpha,b)\rightarrow a\alpha b$ satisfying the identity $(a\alpha b)\beta c=a\alpha(b\beta c)$ for all $a,b,c\in S$ and for all $\alpha,\beta\in\Gamma.$
\end{definition}

\begin{definition}
\cite{I} A $\gamma$-semigroup $S$ is called a Po-$\Gamma$-semigroup if for any $a,b,c\in S$ and $\gamma\in\Gamma,$ $a\leq b$ implies $a\gamma c\leq b\gamma c$ and $c\gamma a\leq c\gamma b.$
\end{definition}

\begin{definition}
Let $S$ be a Po-$\Gamma$-semigroup. By a subsemigroup of $S$ we mean a non-empty subset $A$ of $S$ such that $A\Gamma A\subseteq A.$
\end{definition}

\begin{definition}
Let $S$ be a po-$\Gamma$-semigroup. A subsemigroup $A$ of $S$ is called an interior ideal of $S$ if $(1)$ $S\Gamma A\Gamma S\subseteq A,$ $(2)$ $a\in A$ and $b\leq a$ imply $b\in A.$
\end{definition}

\begin{definition}
\cite{Z} A fuzzy subset $\mu$ of a non-empty set $X$ is a function $\mu:X\rightarrow [0,1].$
\end{definition}

\begin{definition}
\cite{S1} Let $\mu$ be a fuzzy subset of a non-empty set $X.$ Then the set $\mu_{t}=\{x\in X:\mu(x)\geq t\}$ for $t\in [0,1],$ is called the $t$-cut of $\mu.$
\end{definition}

\begin{definition}
A non-empty fuzzy subset $\mu$ of a Po-$\Gamma$-semigroup $S$ is called a fuzzy subsemigroup of $S$ if $\mu(x\gamma y)\geq\min\{\mu(x),\mu(y)\}\forall x,y\in S,\forall\gamma\in\Gamma.$
\end{definition}

\begin{definition}
A fuzzy subsemigroup $\mu$ of a Po-$\Gamma$-semigroup $S$ is called a fuzzy interior ideal of $S$ if $(1)$ $\mu(x\alpha a\beta y)\geq\mu(a)\forall x,a,y\in S,\forall\alpha,\beta\in\Gamma,$ $(2)$ $x\leq y$ implies $\mu(x)\geq\mu(y)\forall x,y\in S.$
\end{definition}


\section{Main Results}

In what follows $Aut(S)$ denote the set of all automorphisms of the Po-$\Gamma$-semigroup $S.$

\begin{definition}
An interior ideal $A$ of a Po-$\Gamma$-semigroup $S$ is called a characteristic interior ideal of $S$ if $f(A)=A$ $\forall f\in Aut(S).$
\end{definition}

\begin{definition}
A fuzzy interior ideal $\mu$ of a Po-$\Gamma$-semigroup $S$ is called a fuzzy characteristic interior ideal of $S$ if $\mu(f(x))=\mu(x)$ $\forall x\in S$ and $\forall f\in Aut(S).$
\end{definition}

\begin{theorem}
A non-empty fuzzy subset $\mu$ of a Po-$\Gamma$-semigroup $S$ is a fuzzy characteristic interior ideal of $S$ if and only if the $t$-cut of $\mu$ is a characteristic interior ideal of $S$ for all $t\in [0,1],$ provided $\mu_{t}$ is non-empty.
\end{theorem}

\begin{proof}
Let $\mu$ be a fuzzy interior ideal of $S$ and $t\in [0,1]$ be such that $\mu_{t}$ is non-empty. Let $x,y\in\mu_{t}$ and $\gamma\in\Gamma.$ Then $\mu(x)\geq t$ and $\mu(y)\geq t.$ Now since $\mu$ is a fuzzy interior ideal, it is a fuzzy subsemigroup of $S$ and hence $\mu(x\gamma y)\geq\min\{\mu(x),\mu(y)\}.$ Thus we see that $\mu(x\gamma y)\geq t.$ Consequently, $x\gamma y\in\mu_{t}.$ Hence $\mu_{t}$ is a subsemigroup of $S.$ Now let $x,y\in S;$ $\beta,\delta\in\Gamma$ and $a\in\mu_{t}.$ Then $\mu(x\beta a\delta y)\geq\mu(a)\geq t$ and so $x\beta a\delta y\in\mu_{t}.$

Let $x,y\in S$ be such that $y\leq x.$ Let $x\in \mu_{t}.$ Then $\mu(x)\geq t.$ Since $\mu$ is a fuzzy interior ideal of $S,$ so $\mu(y)\geq\mu(x)\geq t.$ Consequently, $y\in \mu_{t}.$ Hence we conclude that $\mu_{t}$ is an interior ideal of $S.$

In order to prove the converse, we have to show that $\mu$ satisfies the condition of Definition $2.7$ and the condition of Definition $2.8.$ If the condition of Definition $2.7$ is false, then there exist $x_{0},y_{0}\in S,$ $\gamma\in\Gamma$ such that $\mu(x_{0}\gamma y_{0})<\min\{\mu(x_{0}),\mu(y_{0})\}.$  Taking $t_{0}:=\frac{1}{2}[\mu(x_{0}\gamma y_{0})+\min\{\mu(x_{0}),\mu(y_{0})\}],$ we see that $\mu(x_{0}\gamma y_{0}) <t_{0} <\min\{\mu(x_{0}),\mu(y_{0})\}.$ This implies that $x_{0},y_{0}\in\mu_{t_{0}}$ and $x_{0}\gamma y_{0}\notin\mu_{t_{0}},$ which is a contradiction. Hence the condition of Definition $2.7$ is true. Similarly we can prove the other condition also.

Let $x,y\in S$ be such that $x\leq y.$ Let $\mu(y)= t,$ then $y\in \mu_{t}.$ Since $\mu_{t}$ is an interior ideal of $S,$ so $x\in \mu_{t}.$ Then $\mu(x)\geq t=\mu(y).$ Consequently, $\mu$ is a fuzzy interior ideal of $S.$
\end{proof}

\begin{theorem}
Let $A$ be a non-empty subset of a Po-$\Gamma$-semigroup $S.$ Then $A$ is a characteristic interior ideal of $S$ if and only if its characteristic function $\chi_{A}$ is a fuzzy characteristic interior ideal of $S.$
\end{theorem}

\begin{proof}
Let $A$ be a characteristic interior ideal of $S.$ Then $A$ is an interior ideal of $S.$ Hence by Lemma $1\cite{Kh},$ $\chi_{A}$ is a fuzzy interior ideal of $S.$ Let $x\in S.$ If $x\in A,$ the $\chi_{A}(x)=1.$ Then for all $f\in Aut(S),f(x)\in f(A)=A$ whence $\chi_{A}(f(x))=1.$ If $x\notin A,$ then $\chi_{A}(x)=0.$ Then for all $f\in Aut(S),f(x)\notin f(A)$ whence $\chi_{A}(f(x))=0.$ Thus we see that $\chi_{A}(f(x))=\chi_{A}(x)$ for all $x\in S.$ Hence $\chi_{A}$ is a fuzzy characteristic interior ideal of $S.$

Conversely, let $\chi_{A}$ is a fuzzy characteristic interior ideal of $S.$ Then $\chi_{A}$ is a fuzzy interior ideal of $S.$ Hence by Lemma $1\cite{Kh},$ $A$ is an interior ideal of $S.$ now, let $f\in Aut(S)$ and $a\in A.$ Then $\chi_{A}(a)=1$ and so $\chi_{A}(f(a))=\chi_{A}(a)=1.$ Hence $f(a)\in A.$ Thus, we obtain $f(A)\subseteq A$ for all $f\in Aut(S).$ Again $a\in A,f\in Aut(S)\Rightarrow\exists b\in S\ni f(b)=a.$ If possible, suppose that $b\notin A.$ Then $\chi_{A}(b)=0.$ So $\chi_{A}(f(b))=\chi_{A}(b)=0.$ Hence $f(b)\notin A,i.e.,a\notin A-$ a contradiction. Hence $b\in A,i.e.,f(b)\in A.$ Thus we obtain $A\subseteq f(A)\forall f\in Aut(S).$ Hence the proof.
\end{proof}


\end{document}